\documentclass[reqno,11pt]{amsproc}

\usepackage{graphicx}
\usepackage{tikz}
\usepackage{tikz-cd}
\usepackage[margin=1.5in]{geometry}
\usepackage{amssymb}
\usepackage{hyperref}
\usepackage{microtype}
\usepackage{enumitem}
\usepackage[nocompress]{cite}
\usepackage{bbm}
\usepackage[skip=1ex]{caption}
\usepackage{subcaption}
\usepackage{mathtools}
\usepackage{xpatch}
\usepackage{url}

\newtheorem{theorem}{\bf Theorem}[section]

\newtheorem{lemma}[theorem]{\bf Lemma}
\newtheorem{corollary}[theorem]{\bf Corollary}

\theoremstyle{definition}

\newtheorem*{definition*}{\bf Definition}

\renewcommand{\descriptionlabel}[1]%
     {\hspace{\labelsep}\textsf{#1}}

\makeatletter
\newcommand{\proofpart}[2]{%
  \addvspace{\medskipamount}%
  \noindent{\em Step #1: #2} \par \nobreak
  \addvspace{\smallskipamount}%
  \@afterheading
}
\makeatother

\newcommand		{\p}[1]	{\left(#1\right)}
\newcommand		{\pp}[1]{\left[#1\right]}
\newcommand		{\abs}[1]{\left|#1\right|}

\newcommand     	{\f}  {\mathbb F}
\newcommand     	{\q}  {\mathbb Q}

\renewcommand     	{\c}  {\mathbb C}
\newcommand     	{\gl} {\mathrm{GL}}
\renewcommand     	{\sl} {\mathrm{SL}}
\newcommand     	{\so} {\mathrm{SO}}
\newcommand     	{\su} {\mathrm{SU}}
\renewcommand     	{\sp} {\mathrm{Sp}}
\newcommand     	{\cl} {\overline}
\renewcommand    	{\ker} {\mathrm{Ker}\,}
\renewcommand    	{\dim} {\mathrm{dim}}

\renewcommand    	{\mod} {\mathrm{mod}}
\newcommand    		{\additive} {\mathbb{G}_a}

\DeclareTextFontCommand{\emph}{\bfseries\em}

\begin{document}

\title[Algebraic Groups over Finite Fields: Subgroups and Isogenies]{
	\large Algebraic Groups over Finite Fields: \\
		Connections Between Subgroups and Isogenies}

\author[Davide Sclosa]{Davide Sclosa$^1$}

\date{\today}

\footnotetext[1]{Mathematics Department, Vrije Universiteit Amsterdam;
e-mail:\hfill{\mbox{}}\\\hbox{d.sclosa@vu.nl}}

\keywords{algebraic group, isogeny, rational points, Frobenius, finite index.}

\begin{abstract}
Let~$G$ be a linear algebraic group defined over a finite field~$\f_q$.
We present several connections between the isogenies of~$G$
and the finite groups of rational points~$(G(\f_{q^n}))_{n\geq 1}$.
We show that an isogeny~$\phi: G'\to G$ over~$\f_q$
gives rise to a subgroup of fixed index in~$G(\f_{q^n})$ for infinitely many~$n$.
Conversely, we show that if~$G$ is reductive the existence of a subgroup~$H_n$
of fixed index~$k$ for infinitely many~$n$ implies the existence of an isogeny of order~$k$.
In particular, we show that the infinite sequence~$H_n$ is covered by a finite number of isogenies.
This result applies to classical groups~$\gl_m$, $\sl_m$, $\so_m$, $\su_m$, $\sp_{2m}$
and can be extended to non-reductive groups if~$k$ is prime to the characteristic.
As a special case, we see that if~$G$ is simply connected
the minimal indices of proper subgroups of~$G(\f_{q^n})$
diverge to infinity.
Similar results are investigated regarding the sequence~$(G(\f_p))_p$
by varying the characteristic~$p$.
\end{abstract}

\maketitle


\section{Introduction.}

Linear algebraic groups are groups of matrices defined by polynomial equations.
We adopt the classical notion of algebraic group as a group
of rational points over the algebraic closure, following the language of
A.~Borel, J.~Tits and C.~Chevalley~\cite{borel2001, chevalley1954algebraic, chevalley1947algebraic}.
Throughout the article we clarify
where the modern scheme-theoretic
approach~\cite{milne2017algebraic, RG, waterhouse1979introduction} differs.
We focus on linear algebraic groups over finite fields.
They are are closely related to the classification of finite simple groups~\cite{gorenstein2013}.

Before presenting the main results, we introduce some notation.
Let~$\f_q$ be a finite field with~$q$ elements and characteristic~$p$.
We denote by $\cl{\f_q}$ an algebraic closure. For every~$n\geq 1$
we consider the finite extension~$\f_{q^n} \subseteq \cl{\f_q}$.
A \emph{linear algebraic group}~$G$ defined over~$\f_q$ is
a closed subgroup of~$\gl_m(\cl{\f_q})$ defined by
polynomial equations with coefficients in~$\f_q$.
We always assume that~$G$ is connected.
The \emph{group of} $\f_{q^n}$-\emph{rational points}~$G(\f_{q^n})$
is the subgroup of~$G$ whose elements have entries in~$\f_{q^n}$.
The Frobenius automorphism~$x\mapsto x^{q^n}$ of~$\f_{q^n}$
extends naturally to a group automorphism~$g\mapsto \sigma_{q^n}(g)$ of~$G$
via the action on the matrix entries.
The group of rational points~$G(\f_{q^n})$ is exactly the fixed subgroup of~$\sigma_{q^n}$.
An \emph{isogeny} between connected linear algebraic groups
is a surjective homomorphism~$\phi:G'\to G$ with finite kernel.
We call~\emph{order} of~$\phi$ the cardinality of the kernel~$\abs{\ker\phi}$.

In this paper we present several connections between
the sequence of finite groups~$(G(\f_{q^n}))_{n\geq 1}$ and
the isogenies~$\phi:G'\to G$.
In Theorem~\ref{thm:from_isogenies_to_subgroups} we
show that if~$\phi$ has order~$k$, then for infinitely many~$n$
the group~$G(\f_{q^n})$ contains a subgroup of index~$k$.
More surprisingly, in Theorem~\ref{thm:main} and Theorem~\ref{thm:main_general} we show that,
under suitable hypotheses,
if for infinitely many~$n$ the group~$G(\f_{q^n})$
contains a subgroup of index~$k$,
then there exist a group~$G'$ and an isogeny~$\phi: G'\to G$ of order~$k$.
In particular, we show that finitely many isogenies are responsible for the infinite
sequence of subgroups of index~$k$.

These results constrain the asymptotic behavior of subgroups:
the set of positive integers~$n$ for which~$G(\f_{q^n})$ contains a subgroup of index~$k$
is either finite or contains
an arithmetic progression (Corollary~\ref{cor:arithmetic_progression}).

As a corollary we obtain:
if~$G$ is semi\-simple, simply connected
and~$k>1$, then for every $n$~large enough 
the group~$G(\f_{q^n})$ contains no subgroup of index~$k$
(Corollary~\ref{cor:semisimple_simply_connected}).
Notice that, while our result is purely asymptotic,
the maximal subgroups of simple groups of Lie type
have actually been classified~\cite{kleidman1988survey}.

In the last section, instead of~$(G(\f_{q^n}))_{n\geq 1}$,
we consider the sequence~$(G(\f_p))_p$ by varying the characteristic.
In analogy to Corollary~\ref{cor:semisimple_simply_connected}
we show that if~$G$ is semi\-simple, simply connected
and~$k>1$, then for every prime~$p$ large enough
the group~$G(\f_p)$ contains no subgroup of index~$k$ (Theorem~\ref{thm:characteristic}).


\section*{Acknowledgement.}
This work is part of a {\it Tesi di Laurea} at the University of Udine.
The author is grateful to his advisor, prof. Pietro Corvaja,
for many inspiring discussions, support and patience.
The author would also like to thank prof. Gunter Malle for their time and precious comments
that helped improve the manuscript.


\section{From isogenies to subgroups.}
In this section we show how one rational isogeny gives rise to an infinite family
of subgroups of fixed index.

Let $G'$~and~$G$ be connected linear algebraic groups defined over~$\f_q$.
Let~$\phi: G'\to G$ be an isogeny defined over~$\f_q$.
For every~$n$ the isogeny~$\phi$ restricts to a homomorphism
of finite groups~$\phi: G'(\f_{q^n})\to G(\f_{q^n})$. Notice the abuse of notation.

The kernel of~$\phi: G'(\f_{q^n})\to G(\f_{q^n})$
coincides with the group of~$\f_{q^n}$-rational points of
the kernel of~$\phi: G'\to G$.
The notation~$\ker\phi(\f_{q^n})$ is unambiguous.

The same cannot be said of the image, since
the group~$\phi\p{G'(\f_{q^n})}$ and the group~$\phi(G')(\f_{q^n})=G(\f_{q^n})$
may be different. Indeed, although~$\phi$ is surjective at the level of algebraic closure,
it may not be surjective at the level of rational points:

\begin{lemma} \label{lem:image}
The image~$\phi\p{G'(\f_{q^n})}$ has index~$\abs{\ker\phi(\f_{q^n})}$ in~$G(\f_{q^n})$.
\end{lemma}
\begin{proof}
Since $G'$~and~$G$ are isogenous over~$\f_{q}$, in particular they are isogenous over~$\f_{q^n}$.
Two groups isogenous over~$\f_{q^n}$
have the same number of~$\f_{q^n}$-rational points~\cite[Proposition 16.8]{borel}.
Therefore, the cardinality of the kernel is equal the index of the image.
\end{proof}

In particular, the quotient between two groups of rational points may be different
from the group of rational points of the quotient. This is caused by the discreteness of
the kernel. Indeed, if $N$ is a connected normal subgroup of~$G$ then~$(G/N)(\f_{q^n})$
is equal to~$G(\f_{q^n})/N(\f_{q^n})$, see~\cite[Corollary 16.5 (ii)]{borel}.

From Lemma~\ref{lem:image} we easily obtain the following theorem:

\begin{theorem} \label{thm:from_isogenies_to_subgroups}
Let $G'$~and~$G$ be two connected linear algebraic groups defined over~$\f_q$.
Let~$\phi: G'\to G$ be an isogeny of order~$k$ defined over~$\f_q$.
Then, for infinitely many~$n$, the group of rational points~$G(\f_{q^n})$
has a subgroup of index~$k$.
The set of integers for which it happens contains an arithmetic progression.
\end{theorem}
\begin{proof}
Since~$\ker\phi$ is finite, then~$\ker\phi = \ker\phi(\f_{q^m})$ for some~$m$.
It follows that~$\ker\phi = \ker\phi(\f_{q^n})$ for every~$n$
multiple of~$m$.
By Lemma~\ref{lem:image} the group~$G(\f_{q^n})$ has a subgroup of index~$k$
for every~$n$ multiple of~$m$.
\end{proof}


\section{From subgroups to isogenies: reductive groups.}
In this section we show how the existence of infinitely may subgroups of fixed index~$k$
imply the existence of an isogeny of order~$k$, in the case of reductive groups.

Let~$\phi: G'\to G$ be an isogeny defined over~$\f_q$.
Since~$G'$ is connected
and~$\ker\phi$ is discrete, the action of~$G$ on~$\ker\phi$ by conjugation is trivial.
Therefore, the kernel~$\ker\phi$ is a central subgroup of~$G'$.
Notice, however, that not all isogenies
are central in the scheme-theoretic sense.

We already observed that although
\[
	1\to \ker\phi \to G'\to G \to 1
\]
is exact, the sequence
\[
	1 \to \ker\phi(\f_{q^n}) \to G'(\f_{q^n}) \to G(\f_{q^n}) \to 1
\]
may not be exact. The cokernel~$G(\f_{q^n})/\phi(G'(\f_{q^n}))$
measures how far this sequence is from being exact.
Out next goal is understanding the cokernel.

To this end, we recall a standard tool in the study of algebraic groups over finite fields.
Fix~$n$. The \emph{Lang map}~$\lambda_{q^n}$ is defined by
\[
	\lambda_{q^n} \colon G'\to G', \quad y \mapsto y^{-1}\sigma_{q^n}(y).
\]
Lang's Theorem tells us that this map is surjective~\cite[16.4]{borel}.

Notice that, since~$\phi$ commutes with~$\sigma_{q^n}$,
then~$\lambda_{q^n} (\ker\phi)\subseteq \ker\phi$. Since~$\ker\phi$
is central, in particular commutative,
it follows that~$\lambda_{q^n} (\ker\phi)$ is a normal subgroup of~$\ker\phi$.
Therefore the group quotient~$\ker\phi/{\lambda_{q^n}(\ker\phi)}$ is well defined.
It turns out that it is isomorphic to the cokernel:

\begin{lemma} \label{lem:iso}
The following group isomorphism holds:
$$ \frac{G(\f_{q^n})}{\phi(G'(\f_{q^n}))} \cong \frac{\ker\phi}{\lambda_{q^n}(\ker\phi)}.$$
In particular~$\phi(G'(\f_{q^n}))$ is a normal subgroup of~$G(\f_{q^n})$.
\end{lemma}
\begin{proof}
Take~$x\in G(\f_{q^n})$. Since~$\phi:G'\to G$ is surjective, then there is~$y\in G'$
such that~$\phi(y)=x$. Since~$\phi$ commutes with~$\sigma_{q^n}$ and~$\sigma_{q^n}(x)=x$
we have~$y^{-1}\sigma_{q^n}(y) \in \ker\phi$. Consider the map
\begin{equation*}
	\mu_{q^n} \colon G(\f_{q^n}) \to \frac{\ker\phi}{\lambda_{q^n}(\ker\phi)},
	\quad x\mapsto \lambda_{q^n} (y) .
\end{equation*}
First of all, we need to check that~$\mu_{q^n}$ is well defined.
If a different~$y$ is chosen, say~$z\in G'$ such that~$\phi(z)=x$,
then~$yz^{-1} \in \ker\phi$ and
so~$(yz^{-1})^{-1} \sigma_{q^n}(yz^{-1}) \in \lambda_{q^n}(\ker\phi)$,
which is equivalent
to~$(z^{-1}\sigma_{q^n}(z))^{-1} y^{-1}\sigma_{q^n}(y) \in \lambda_{q^n}(\ker\phi)$
since~$\lambda_{q^n}(\ker\phi)$ is central.

Now we prove that~$\mu_{q^n}$ is surjective. Let~$a\in \ker\phi$.
By Lang's Theorem
there is~$y\in G'$ such that~$y^{-1}\sigma_{q^n}(y) = a$.
Let~$x = \phi(y)$. Since~$\phi(y^{-1}\sigma_{q^n}(y)) = 1$ we have~$\sigma_{q^n}(x)=x$
and so~$x\in G(\f_{q^n})$. By definition~$x$ is mapped to~$y^{-1}\sigma_{q^n}(y)$,
which is equal to~$a$.

Next, we prove that~$\mu_{q^n}$ is a homomorphism.
Take any~$x,w\in G(\f_{q^n})$, let~$\phi(y)=x$ and~$\phi(z)=w$.
We need to show that~$(yz)^{-1}\sigma_{q^n}(yz)$ is equal
to~$y^{-1}\sigma_{q^n}(y) z^{-1}\sigma_{q^n}(z)$. This is the same
as~$z^{-1} y^{-1}\sigma_{q^n}(y) = y^{-1} \sigma_{q^n}(y) z^{-1}$, which holds
since~$y^{-1}\sigma_{q^n}(y)\in \ker\phi$ is central.

Finally, we prove that the kernel of~$\mu_{q^n}$ is~$\phi(G'(\f_{q^n}))$.
Let~$x = \phi(y)$ be an element of~$G(\f_{q^n})$ such that~$y^{-1}\sigma_{q^n}(y)$
belongs to~$\lambda_{q^n}(\ker\phi)$. Then~$y^{-1}\sigma_{q^n}(y) = a^{-1}\sigma_{q^n}(a)$ for
some~$a\in \ker\phi$. Since~$\ker\phi$ is
central and~$y^{-1}\sigma_{q^n}(y) = a^{-1}\sigma_{q^n}(a)$ it follows that
$$ \sigma_{q^n}(a^{-1} y) = \sigma_{q^n}(y) \sigma_{q^n}(a)^{-1} = a^{-1} y$$
and so~$a^{-1}y \in G'(\f_{q^n})$.
Therefore~$x = \phi(y) = \phi(a^{-1}y)$ belongs to~$\phi(G'(\f_{q^n}))$.
\end{proof}

Let~$H$ be a subgroup of~$G(\f_{q^n})$. We say that an isogeny~$\phi:G'\to G$
\emph{reaches}~$H$ if it is defined over~$\f_{q^n}$
and~$\phi(G'(\f_{q^n})) = H$.
In the following lemma, which serves as a bootstrap,
we show how to construct an isogeny reaching~$H$ from an isogeny whose image
is contained in~$H$.
The idea is simple: in light of Lemma~\ref{lem:image}, in order to make the image larger,
we need to make the quotient smaller.

\begin{lemma} \label{lem:constructing_the_isogeny}
Let $G'$,~$G$~and~$\phi:G'\to G$ be two connected linear algebraic groups
and an isogeny defined over~$\f_{q^n}$.
Let~$H$ be a subgroup of~$G(\f_{q^n})$ containing~$\phi(G'(\f_{q^n}))$ and let~$K=\mu_{q^n}(H)$. Then~$G''=G'/K$ is a connected linear algebraic group
defined over~$\f_{q^n}$; the induced isogeny~$\phi: G''\to G$ reaches~$H$.
\end{lemma}
\begin{proof}
The group~$K_{q^n}$ is defined over~$\f_{q^n}$:
since~$\lambda_{q^n} (\ker\phi) \subseteq K_{q^n} \subseteq \ker\phi$,
in particular~$\lambda_{q^n} (K_{q^n}) \subseteq K_{q^n}$ and so~$K_{q^n}$
is~$\sigma_{q^n}$-invariant.
Therefore the quotient group~$G''=G'/K_{q^n}$
is defined over~$\f_{q^n}$. Its group of~$\f_{q^n}$-rational points is
\begin{align*}
G'' (\f_{q^n}) &= \p{G'/K_{q^n}}\p{\f_{q^n}} \\
&= \{yK_{q^n} : \quad y\in G', \quad \sigma_{q^n}(yK_{q^n}) = yK_{q^n} \} \\
&= \{yK_{q^n} : \quad y\in G', \quad \lambda_{q^n}(y) \in K_{q^n} \}.
\end{align*}
Since~$K_{q^n}\subseteq \ker\phi$ the isogeny~$\phi:G'\to G$
induces a well defined isogeny~$\phi: G''\to G$.
Moreover
$$\phi(G''(\f_{q^n}))
= \{\phi(y) : \quad y\in G', \quad \lambda_{q^n}(y) \in K_{q^n} \} = H_n, $$
completing the proof.
\end{proof}

Lemma~\ref{lem:constructing_the_isogeny} helps in constructing one isogeny
reaching one subgroup. One expects that an infinite family of subgroups
require infinitely many isogenies. However:

\begin{corollary} \label{cor:finitely_many_isogenies}
Let $G'$,~$G$~and~$\phi:G'\to G$ be two connected linear algebraic groups
and an isogeny defined over~$\f_{q^n}$.
For infinitely many~$n$, let~$H_n$
be a subgroup of~$G(\f_{q^n})$ containing~$\phi(G'(\f_{q^n}))$.
Then there are finitely many isogenies reaching all~$H_n$.
\end{corollary}
\begin{proof}
Fix any~$n$ for which~$H_n$ is defined.
Since~$G$,~$G'$ and~$\phi$ are defined over~$\f_q$, in particular they are defined
over~$\f_{q^n}$, so Lemma~\ref{lem:constructing_the_isogeny} applies.

Now let~$n$ vary. Since~$K_{q^n}\subseteq \ker\phi$ and~$\ker\phi$ is finite, there are only
finitely many possibilities for~$\phi: G'' \to G$.
\end{proof}

Lemma~\ref{lem:constructing_the_isogeny} and Corollary~\ref{cor:finitely_many_isogenies}
are key in the proof of the main theorems. We also need the following elementary lemma,
whose proof is left as an exercise:

\begin{lemma} \label{lem:index}
Let~$G$ be a finite group. Let $N$~and~$H$ be two subgroups.
Suppose that~$N$ is normal. Then
\[
	\bigl[G : H\bigr] = \pp{\frac{G}{N} : \frac{HN}{N} } \bigl[N : H\cap N\bigr].
\]
\end{lemma}

Following~\cite{malle} and~\cite{rapinchuk} we call~\emph{simple} a semisimple 
algebraic group with no proper positive dimensional normal subgroup.
Notice that other authors prefer the name~\emph{almost-simple}~\cite{RG}.

The main result about simple groups over finite fields is due to J.~Tits:
let $G$ be a simple, simply connected linear algebraic
group defined over~$\f_{q}$. Unless~$G(\f_{q^n})$ is one of
$$ \sl_2(\f_2),\ \sl_2(\f_3),\ \su_3(\f_2),\
	\sp_4(\f_2),\ G_2(\f_2),\ \prescript{2}{}{B_2(\f_2)},\
	\prescript{2}{}{G_2(\f_3)},\ \prescript{2}{}{F_4(\f_2)} $$
the group $G(\f_{q^n})/Z(G(\f_{q^n}))$ is simple~\cite[Theorem 24.17]{malle}.
We refer to this result as Tits'~Theorem in the remainder of the paper.
Notice that last three groups in the list are not groups of rational points:
they are not fixed subgroups of a Frobenius automorphism, but of a Steinberg endomorphism~\cite[Definition 21.3]{malle}.
Therefore, the list of exceptions is actually shorter.
The only thing we will need is that it is finite. 

\begin{theorem} \label{thm:main}
Let~$G$ be reductive linear algebraic groups defined over~$\f_q$.
Let~$k\geq 1$ be such that for infinitely many~$n$ the group of rational points~$G(\f_{q^n})$
contains a subgroup~$H_n$ of index~$k$.
Then there are finitely many linear algebraic groups~$G'$
and isogenies~$G' \to G$ such that, except possibly for finitely many~$n$,
every~$H_n$ is reached by one of them.
\end{theorem}

\begin{proof}
Reductive groups can be obtained from simple groups and tori by taking
products and isogenies. 
The proof of the theorem, which consists of several steps, shows that
the class of algebraic groups satisfying the statement contains simple groups and tori
and it is closed under the formation of reductive groups.

\proofpart{1}{If~$G$ is simple, simply connected, then~$k=1$.}
Let~$G$ be simple, simply connected and suppose~$k>1$. Choose~$n$
such that~$H_n$ is defined. Since~$H_n$ has index~$k$ then its normalizer
has index at most~$k$, therefore~$H_n$ has at most~$k$ conjugates.
Each of them has index~$k$, so their intersection
$$ \bigcap \big\{ H_n^g : \quad g\in G(\f_{q^n}) \big\}$$
is a normal subgroup of index between~$k$ and~$k^k$.
Since~$k>1$ this implies that the subgroup is proper.
For~$n$ large enough Tits' Theorem implies
$$ \frac{\abs{G(\f_{q^n})}}{\abs{Z(G(\f_{q^n}))}} \leq k^k. $$

On the other hand, being connected the group~$G$ is an irreducible variety and so
by Lang-Weil Theorem~\cite[Theorem 1]{lang-weil}
we have~$\abs{G(\f_{q^n})}\sim q^{n\ \dim G}$ for~$n\to \infty$.
By hypothesis~$ZG$ is finite and
$Z(G(\f_{q^n}))=(ZG)(\f_{q^n})$ by~\cite[Proposition 24.13]{malle},
implying
$$ \lim_{n\to\infty} \frac{\abs{G(\f_{q^n})}}{\abs{Z(G(\f_{q^n}))}} = \infty$$
which is a contradiction. We conclude that~$k=1$.
The statement is trivially satisfied by the identity map~$G\to G$.

\proofpart{2}{If~$G$ is semisimple, simply connected, then~$k=1$.}
Let~$G$ be semisimple, simply connected. Then~$G$ is direct product
of its minimal connected normal subgroups which are simple,
simply connected and defined
over~$\f_q$~\cite[17.22, 17.24]{RG}.
We proceed by induction
on the number simple components. Let~$N$ be one simple component.
By Step~1, applied to~$N$, the intersection~$H_n \cap N(\f_{q^n})$
must be equal to~$N(\f_{q^n})$ for~$n$ large enough. By Lemma~\ref{lem:index}
the index~$k=\pp{G(\f_{q^n}):H_n}$ must be equal to the index of~$H_n /N(\f_{q^n})$
in~$G(\f_{q^n})/N(\f_{q^n})$.
Since~$N$ is connected we have~$(G/N)(\f_{q^n}) = G(\f_{q^n})/N(\f_{q^n})$
by~\cite[Corollary 16.5 (ii)]{borel} and induction applies to~$G/N$.

\proofpart{3}{The statement is true if~$G$ is semisimple.}
Let~$G$ be semisimple. Then it admits a
universal covering~$\phi: G' \to G$, where~$G'$
is simply connected and~$\phi$ is an isogeny, both defined over~$\f_q$~\cite[16.21]{RG}.
Note that since~$G$ is semisimple then~$G'$ is semisimple and so Step 2 applies to~$G'$.
The preimages of~$\phi(G(\f_{q^n}))\cap H_n$ with respect to~$\phi$ have index bounded by~$k$ and so
by Step 2 they coincide with~$G'(\f_{q^n})$ for~$n$ large enough.
This means that~$\phi(G'(\f_{q^n})) \subseteq H_n$.
Corollary~\ref{cor:finitely_many_isogenies} applies.

\proofpart{4}{The statement is true if~$G$ is a torus.}
Let~$G$ be a torus. Since~$G(\f_{q^n})$ is isomorphic
to a subgroup of~$(\f^*_{q^m})^d$ for some~$m$~and~$d$, then the integers~$k$~and~$q$
must be coprime.
Consider the homomorphism~$G\to G$ sending~$g\mapsto g^k$,
it is defined over~$\f_q$. Since~$k$~and~$q$
are coprime this morphism is an isogeny. Since~$G(\f_{q^n})/H_n$ has~$k$ elements
we have~$G(\f_{q^n})^k \subseteq H_n$. Corollary~\ref{cor:finitely_many_isogenies} applies.

\proofpart{5}{The statement is true if~$G$ is the direct product of a torus and
a semisimple group, both defined over~$\f_q$.}
Let~$G=T\times S$ where~$T$ is a torus and~$S$ is semisimple.
The integers~$[T(\f_{q^n}) : H_n\cap T(\f_{q^n})]$
and~$[S(\f_{q^n}) : H_n\cap S(\f_{q^n})]$ are divisors of~$[G(\f_{q^n}) : H_n]=k$.
Since the divisors of~$k$ are finite in number, we reduce to finitely
many cases, so by Step 3 and Step 4 we find
finitely many isogenies~$\phi\times\psi$
reaching~$\p{H_n\cap T(\f_{q^n})}\times \p{H_n\cap S(\f_{q^n})}$,
which is a subgroup of~$H_n$. One more time, Corollary~\ref{cor:finitely_many_isogenies} applies.

\proofpart{6}{The statement is true if~$G$ is reductive.}
Let~$G$ be reductive. Let~$T$ be the identity component
of~$ZG$ and~$S=[G,G]$ the derived subgroup.
Since~$\f_q$ is perfect~$T$ and~$S$ are defined over~$\f_q$~\cite[12.1.7]{springer}.
The group~$T$ is a torus, the group~$S$ is semisimple and the product map~$\pi\colon T\times S\to G$
is an isogeny~\cite[Proposition 6.20, Corollary 8.22]{malle}.

By Lemma~\ref{lem:index} the index of~$H_n \cap T(\f_{q^n})S(\f_{q^n})$ in~$T(\f_{q^n})S(\f_{q^n})$
divides~$k$, so there are only finitely many possibilities and hence
finitely many possibilities
for the index of~$\pi^{-1} (H_n)$ in~$T(\f_{q^n})\times S(\f_{q^n})$.
By Step 5 for every~$\pi^{-1} (H_n)$
we find an isogeny~$\phi\colon G' \to T\times S$ defined over~$\f_{q^n}$
reaching~$\phi(G'(\f_{q^n}))=\pi^{-1} (H_n)$
and finitely many isogenies are enough to reach all~$\pi^{-1} (H_n)$.
For every~$\phi$ the composition~$\pi \circ \phi \colon G'\to G$
is an isogeny defined over~$\f_{q^n}$ and
the image~$\pi(\phi(G'(\f_{q^n}))$ is contained in $H_n$.
Lemma~\ref{lem:constructing_the_isogeny} applies to each~$\phi$,
giving finitely many isogenies reaching all~$H_n$.
\end{proof}

Some parts of the proof are interesting in their own right.
From Step~2 we have:

\begin{corollary} \label{cor:semisimple_simply_connected}
Let~$G$ be semisimple, simply connected and~$k>1$.
Then, for~$n$ sufficiently large, the group~$G(\f_{q^n})$ contains no subgroups of index~$k$.
In particular, the minimal indices of proper subgroups of~$G(\f_{q^n})$ diverge to infinity.
\end{corollary}
\begin{proof}
Fix~$h>1$. For every~$k\in \{2,\ldots,h\}$, there is~$n_k$ such that
for every~$n>n_k$ the group~$G(\f_{q^n})$ contains no subgroups of index~$k$.
For every~$n$ larger than~$\max\{n_2,\ldots,n_h\}$ the minimal
index of proper subgroups of~$G(\f_{q^n})$ is larger than~$h$.
\end{proof}

From Step~3 we have:

\begin{corollary}
Let~$G$ be semisimple and let~$d$ be the order of its universal covering~$\phi:G'\to G$.
Let~$H_n \subseteq G(\f_{q^n})$ be an infinite sequence of subgroups of fixed index~$k$.
Then for every~$n$ large enough the group~$H_n$ contains the image of the universal
covering~$\phi(G'(\f_{q^n}))$. In particular~$k\leq d$.
\end{corollary}

By combining Theorem~\ref{thm:from_isogenies_to_subgroups} and Theorem~\ref{thm:main}
we obtain a remarkable fact:

\begin{corollary} \label{cor:arithmetic_progression}
Let~$G$ be reductive and~$k\geq 1$.
The set of~$n$ such that~$G(\f_{q^n})$ contains a subgroup of index~$k$
is either finite, or it contains an arithmetic progression.
\end{corollary}


\section{From subgroups to isogenies: non-reductive groups.}
Theorem~\ref{thm:main} requires~$G$ to be reductive.
The hypothesis is necessary, due to the abundance of $p$-subgroups in unipotent groups.

For example, consider the additive group~$\additive$ defined over~$\f_p$.
We have~$\additive(\f_{p^n})=\f_{p^n}$, endowed with the field addition.
The subgroups of index~$p$ of~$\f_{p^n}$ are the hyperplanes of~$\f_{p^n}$,
as a vector space over~$\f_p$. There are~$\p{p^n-1}/\p{p-1}$ of them.
In particular, the subgroups of index~$p$ grow in number with~$n$.
They cannot all be reached by finitely many isogenies.

%
%

However, if~$k$ is prime to the characteristic, then reductiveness is not necessary:

\begin{theorem} \label{thm:main_general}
Let~$G$ be a linear algebraic group defined over~$\f_q$.
Let~$k\geq 1$ be prime to~$q$ and
such that for infinitely many~$n$ the group of rational points~$G(\f_{q^n})$
contains a subgroup~$H_n$ of index~$k$.
Then there are finitely many linear algebraic groups~$G'$
and isogenies~$G' \to G$ such that, except for finitely many~$n$,
every~$H_n$ is reached by one of them.
\end{theorem}
\begin{proof}
Let~$U$ be the unipotent radical of~$G$.
Since~$\f_q$ is perfect then~$U$ is defined over~$\f_q$~\cite[12.1.7]{springer}.
Since~$U(\f_{q^n})$ is a~$p$-group
then the index of~$U(\f_{q^n})\cap H_n$ in~$U(\f_{q^n})$
divides both~$q$~and~$k$, so it must be~$1$ and hence~$U(\f_{q^n})\subseteq H_n$.
Theorem~\ref{thm:main} applied
to the quotient~$G/U$ gives finitely many isogenies~$\phi:G'\to G/U$
reaching~$H_n/U(\f_{q^n})$ for~$n$ large.

Fix~$n$ and one isogeny~$\phi:G'\to G/U$ reaching~$H_n/U(\f_{q^n})$.
We want to lift~$\phi$ to~$G$.
Consider the fiber product~$G'\times_{G/U} G$
whose elements are the pairs $(g',g)\in G'\times G$
satisfying~$\phi(g')=\pi(g)$, where~$\pi:G\to G/U$ is the canonical projection.
We have the following commutative diagram
$$
\begin{tikzcd}
G'\times_{G/U} G \arrow[r] \arrow[d]
	& G \arrow[d, "\pi"] \\
G' \arrow[r, "\phi"]	& G/U
\end{tikzcd}
$$
where the unlabelled arrows are the canonical projections.

For now, suppose that~$G'\times_{G/U} G$ is connected.
The top arrow~$\phi': G'\times_{G/U} G\to G$ is an isogeny since its kernel
is~$\ker(\phi)\times 1$. Since~$\phi$,~$\pi$ and~$U$ are defined over~$\f_q$
and~$U$ is connected we have
$$(G'\times_{G/U} G) (\f_{q^n}) = G'(\f_{q^n}) \times_{G(\f_{q^n})/U(\f_{q^n})} G(\f_{q^n})$$
and by commutativity of the diagram
$$ \pi(\phi'((G'\times_{G/U} G) (\f_{q^n}))) = \phi(G'(\f_{q^n})). $$
The term on the right equals~$H_n/U(\f_{q^n})$. We conclude that~$\phi'((G'\times_{G/U} G) (\f_{q^n})))=H_n$.
Finally,
if~$G'\times_{G/U} G$ is not connected replace it by its identity component:
since~$G$~and~$G'$ are connected the projections remain surjective and the
same argument applies.
\end{proof}


\section{Varying the characteristic.}
In this section~$G$ is a linear algebraic group defined over~$\q$.
Except for finitely may primes the group~$G$ is well-defined module~$p$
and we can consider~$G(\f_p)$.
In analogy with the previous result, we ask whether the subgroups
in the sequence~$(G(\f_p))_p$ are related to isogenies~$\phi:G'\to G$
defined over~$\q$.

Corollary~\ref{cor:semisimple_simply_connected} about simply connected groups
has a perfect analogue:

\begin{theorem} \label{thm:characteristic}
Let~$G$ be a semisimple, simply connected linear algebraic group defined over~$\q$.
Let~$k>1$. Then for~$p$ sufficiently large the group~$G(\f_p)$ contains
no subgroup of index~$k$.
\end{theorem}

\begin{proof}
Suppose that~$G$ is simple and simply connected, the general case follows by induction
on the number of simple factors. The proof is by contradiction: Suppose
that for infinitely many~$p$ the group~$G(\f_p)$ contains a subgroup of index~$k>1$.
As in the proof of Theorem~\ref{thm:main} for~$p$ large enough we have
$$ \frac{\abs{G(\f_{p})}}{\abs{Z(G(\f_{p}))}} \leq k^k. $$
In the proof of Theorem~\ref{thm:main}
Lang-Weil Theorem led from here to a contradiction.
Here we reach a contradiction by applying the theory of $BN$-pairs~\cite[Chapter 11]{malle}.
Let~$T$ be a maximal torus of~$G$ defined over~$\f_p$. Let~$W$ be its Weyl group
and~$\Phi^+$ be the set of positive roots. We have
\[
	\abs{G(\f_p)} = 
		p^{\abs{\Phi^+}}	\abs{T(\f_p)}	\sum_{w\in W(\f_p)} p^{l(w)}
\]
where~$l(w)$ are positive integers independent of~$p$~\cite[Proposition 24.3]{malle}.
Since~$ZG \subseteq T$ and~$Z(G(\f_p))=(ZG)(\f_p)$~\cite[Proposition 24.13]{malle} we have
$$ \frac{\abs{G(\f_p)}}{\abs{Z(G(\f_{p}))}}
	\geq p^{\abs{\Phi^+}}  \frac{\abs{T(\f_p)}}{\abs{Z(G(\f_{p}))}}	\sum_{w\in W(\f_p)} p^{l(w)}
	\geq p. $$
We obtain~$p\leq k^k$ for infinitely many primes, a contradiction.
\end{proof}

We now show, by an explicit example, that infinitely many subgroups
of fixed index may be unreachable by a rational isogeny.
Consider the torus
\[ G = \left\{
\begin{pmatrix}
a & -b \\
b & a-b
\end{pmatrix} \mid a^2-ab+b^2\neq 0
\right\}.
\]
Notice that~$G(\f)$ is split if and only if the field~$\f$
contains a primitive third root of unity. Indeed, let~$\xi^3=1$ and~$\xi\neq 1$.
We have~$a^2-ab+b^2 = (a + \xi b)(a + \xi^2 b)$ and
\[
\begin{pmatrix}
1 & \xi \\
1 & \xi^2
\end{pmatrix}
\begin{pmatrix}
a & -b \\
b & a-b
\end{pmatrix}
\begin{pmatrix}
1 & \xi \\
1 & \xi^2
\end{pmatrix}^{-1} =
\begin{pmatrix}
a + \xi b & 0 \\
0 & a + \xi^2 b
\end{pmatrix}.
\]

As variety over~$\c$ the group~$G(\c)$ is equal to the projective plane
minus three lines: the line at infinity, which is defined over~$\q$,
and two conjugated lines~$a + \xi b$~and~$a + \xi^2 b$.
This gives three $2$-fold coverings of which only one can be defined over~$\q$.
The rational covering actually corresponds to a rational isogeny, namely
\[
\begin{pmatrix}
a & -b & 0\\
b & a-b & 0 \\
0 & 0 & c
\end{pmatrix} \mapsto
\begin{pmatrix}
a & -b \\
b & a-b
\end{pmatrix}
\]
where~$c^2=a^2-ab+b^2$. Matrices as the one on the left form a two-dimensional torus.

On the other hand,
for every prime~$p$ satisfying~$p\equiv 1\ \mod \ 3$
the field~$\f_p$ contains a primitive third root of unity. Therefore,
the group~$G(\f_p)$ is
isomorphic to~${\f_p}^*\times{\f_p}^*$.
In particular, it has three subgroups of index~$2$.
We deduce that at least two subgroups cannot be reached by rational isogenies.

However, this suggests that infinitely many subgroups of fixed index
in~$(G(\f_p))_p$ correspond to isogenies over a finite
extension of~$\q$. We leave this fact as a conjecture.

\bibliographystyle{siam}
\bibliography{refs}

\begin{thebibliography}{10}

\bibitem{borel2001}
{\sc A.~Borel}, {\em Essays in the history of Lie groups and algebraic groups},
  American Mathematical Soc., 2001.

\bibitem{borel}
\leavevmode\vrule height 2pt depth -1.6pt width 23pt, {\em Linear algebraic
  groups}, vol.~126, Springer Science \& Business Media, 2012.

\bibitem{chevalley1947algebraic}
{\sc C.~Chevalley}, {\em Algebraic lie algebras}, Annals of Mathematics,
  (1947), pp.~91--100.

\bibitem{chevalley1954algebraic}
\leavevmode\vrule height 2pt depth -1.6pt width 23pt, {\em On algebraic group
  varieties}, Journal of the Mathematical Society of Japan, 6 (1954),
  pp.~303--324.

\bibitem{gorenstein2013}
{\sc D.~Gorenstein}, {\em Finite simple groups: An introduction to their
  classification}, Springer Science \& Business Media, 2013.

\bibitem{malle}
{\sc D.~T. Gunter~Malle}, {\em Linear Algebraic Groups and Finite Groups of Lie
  Type}, Cambridge Studies in Advanced Mathematics 133, Cambridge University
  Press, 2011.

\bibitem{kleidman1988survey}
{\sc P.~B. Kleidman and M.~W. Liebeck}, {\em A survey of the maximal subgroups
  of the finite simple groups}, Geometriae Dedicata, 25 (1988), pp.~375--389.

\bibitem{lang-weil}
{\sc S.~Lang and A.~Weil}, {\em Number of points of varieties in finite
  fields}, American Journal of Mathematics, 76 (1954), pp.~819--827.

\bibitem{milne2017algebraic}
{\sc J.~S. Milne}, {\em Algebraic groups: the theory of group schemes of finite
  type over a field}, vol.~170, Cambridge University Press, 2017.

\bibitem{RG}
\leavevmode\vrule height 2pt depth -1.6pt width 23pt, {\em {Reductive Groups}}.
\newblock \url{https://www.jmilne.org/math/CourseNotes/RG.pdf}, 2018.
\newblock Accessed: March 10, 2018, v2.00.

\bibitem{rapinchuk}
{\sc V.~Platonov, A.~Rapinchuk, and R.~Rowen}, {\em Algebraic groups and number
  theory}, Academic press, 1993.

\bibitem{springer}
{\sc T.~Springer}, {\em Linear Algebraic Groups}, Modern Birkhäuser Classics,
  Birkhäuser Basel, 2~ed., 1998.

\bibitem{waterhouse1979introduction}
{\sc W.~Waterhouse}, {\em Introduction to Affine Group Schemes}, Graduate Texts
  in Mathematics, Springer, 1979.

\end{thebibliography}

\end{document}